\documentclass[10pt, reqno]{amsart}
\usepackage{amsmath, amsthm, amscd, amsfonts, amssymb, graphicx, color}
\usepackage[bookmarksnumbered, colorlinks, plainpages]{hyperref}
\hypersetup{colorlinks=true,linkcolor=red, anchorcolor=green, citecolor=cyan, urlcolor=red, filecolor=magenta, pdftoolbar=true}
\usepackage{mathrsfs}

\newtheorem{theorem}{Theorem}[section]
\newtheorem{lemma}[theorem]{Lemma}

\theoremstyle{definition}

\theoremstyle{remark}
\newtheorem{remark}[theorem]{Remark}
\numberwithin{equation}{section}

\usepackage{tikz}

\begin{document}
\setcounter{page}{1}

\title[Subfactors and Hecke groups]{Subfactors and Hecke groups}

\author[A.Glubokov]{Andrey Glubokov$^1$}

\author[I.Nikolaev]{Igor Nikolaev$^2$}


\address{$^{1}$ Department of Mathematics,  Ave Maria University, 
5050 Ave Maria Blvd.,  Ave Maria, FL 34142,  United States.}
\email{\textcolor[rgb]{0.00,0.00,0.84}{andrey.glubokov@avemaria.edu}}

\address{$^{2}$ Department of Mathematics and Computer Science, St.~John's University, 8000 Utopia Parkway,  
New York,  NY 11439, United States.}
\email{\textcolor[rgb]{0.00,0.00,0.84}{igor.v.nikolaev@gmail.com}}





\begin{abstract}
We study a relation between the Hecke groups and the  index of  subfactors in 
a von Neumann algebra. Such a problem was raised  by V.~F.~R.~Jones. 
We solve the problem using the notion of a cluster $C^*$-algebra. 
\end{abstract}

\maketitle

\section{Introduction}
The following  problem can be found  in  [Jones 1991] \cite[p.24]{J}:

\bigskip
\begin{minipage}[5cm]{11cm}
\textit{
``Consider the subgroup $G_{\lambda}$ of $SL_2(\mathbf{R})$ generated by 
$\left(\begin{smallmatrix} 1 & \lambda\cr 0 & 1\end{smallmatrix}\right)$
and  $\left(\begin{smallmatrix} 0 & 1\cr -1 & 0\end{smallmatrix}\right)$.
For what values of $\lambda>0$ is it discrete?  Answer:
$\lambda=2\cos \left({\pi\over n}\right), ~n=3,4,\dots$  or 
$\lambda\ge 2$.  (...)  We have been unable to find any 
direct connection between this result and  Theorem 3.1
(Jones Index Theorem).  It is a tantalizing situation.''  
}
\end{minipage}

\bigskip\noindent
The aim of our note is to solve the problem in terms of the cluster 
$C^*$-algebras \cite[Section 4.4.3]{Nik1}. To give an idea,  let $\mathscr{D}=\{z=x+iy\in\mathbf{C} ~|~ r\le |z|\le R\}$ 
be an annulus in the complex plane. 
Consider the  Schottky  uniformization of $\mathscr{D}$,  i.e. 
\begin{equation}\label{eq1.1}
\mathscr{D}\cong \mathbf{C}P^1 / A^{\mathbf{Z}},  
\end{equation}
where $\mathbf{C}P^1:=\mathbf{C} \cup \{\infty\}$ is the Riemann sphere and 
$A\in SL_2(\mathbf{C})/\pm I$ is a matrix acting on the   
$\mathbf{C}P^1$ by the M\"obius transformation. 
It follows from [Glubokov \& Nikolaev 2018] \cite{GluNik1} and Section 2.2,  that 
the index of subfactors in a von Neumann algebra coincides with the square of 
trace of matrix $A$,  i.e. 
\begin{equation}\label{eq1.2}
tr^2 ~(A) \in [4, \infty) ~\bigcup  ~\{ 4\cos^2\left({\pi\over n}\right) ~|~n\ge 3\}.
\end{equation}

\medskip
To solve the Jones Problem,  we prove in Section 3  that $\mathscr{D}$
 is a ramified double cover of the  orbifold  $\mathbb{H}/G_{\lambda}$,
 where  $\mathbb{H}:=\{x+iy\in \mathbf{C} ~|~ y>0\}$ is  the Lobachevsky half-plane
 and the group $G_{\lambda}$ acts on $\mathbb{H}$ by the linear fractional trasformations.
 Since such a cover takes  the square root of the moduli 
parameter  $tr^2 ~(A)$ of $\mathscr{D}$, we conclude that $\lambda=tr~(A)$.
In other words,  the Jones Index Theorem (\ref{eq1.2})  is equivalent to the following well-known result:
\begin{theorem}\label{thm1.1}
{\bf ([Hecke 1936] \cite[Satz 1,2 \& 6]{Hec1})}
The $G_{\lambda}$ is a discrete subgroup of $SL_2(\mathbf{R})$
if and only if
\begin{equation}\label{eq1.3}
\lambda\in [2, \infty) ~\bigcup  ~\{ 2\cos \left({\pi\over n}\right) ~|~n\ge 3\}.
\end{equation}
\end{theorem}

\medskip
\begin{remark}\label{rmk1.1}
The group $G_{\lambda}$  appears in the study of the
Riemann zeta  function [Hecke 1936] \cite{Hec1}.  In particular, Hecke's Theorem  
says that the space of automorphic 
functions corresponding to $G_{\lambda}$ is (i) infinite-dimensional,
if $\lambda>2$ or (ii)  finite-dimensional,  if 
$\lambda\in\{ 2\cos \left({\pi\over n}\right) ~|~n\ge 3\}$. 
The proof of this fact  is purely analytic [Hecke 1936] \cite{Hec1}. 
On the other hand,  cases  (i) and (ii)  follow 
from the Sherman-Zelevinsky Theorem for the cluster
$C^*$-algebras of rank 2, see  
[Glubokov \& Nikolaev 2018] \cite{GluNik1}. 
  \end{remark}
The article is organized as follows. Section 2 contains a brief review
of preliminary results.  Theorem \ref{thm1.1} is proved  in Section 3.

\section{Preliminaries}
The cluster $C^*$-algebras and their $K$-theory are covered in 
 \cite[Section 4.4.3]{Nik1}.  A correspondence between the cluster
 $C^*$-algebra of an annulus $\mathscr{D}$ and the Jones Index Theorem
 was established in [Glubokov \& Nikolaev 2018] \cite{GluNik1}.
 The Hecke groups were introduced in  [Hecke 1936] \cite{Hec1}.

\subsection{Cluster $C^*$-algebras}
A {\it cluster algebra}  $\mathscr{A}(\mathbf{x}, B)$ of rank $n$ 
is a subring of the field  of  rational functions in $n$ variables
depending  on a  cluster  of variables  $\mathbf{x}=(x_1,\dots, x_n)$
and a skew-symmetric matrix  $B=(b_{ij})\in M_n(\mathbf{Z})$; 
the pair  $(\mathbf{x}, B)$ is called a  seed.
A new cluster $\mathbf{x}'=(x_1,\dots,x_k',\dots,  x_n)$ and a new
skew-symmetric matrix $B'=(b_{ij}')$ is obtained from 
$(\mathbf{x}, B)$ by the   exchange relations:
\begin{eqnarray}\label{eq2.1}
x_kx_k'  &=& \prod_{i=1}^n  x_i^{\max(b_{ik}, 0)} + \prod_{i=1}^n  x_i^{\max(-b_{ik}, 0)},\cr 
b_{ij}' &=& 
\begin{cases}
-b_{ij}  & \mbox{if}   ~i=k  ~\mbox{or}  ~j=k\cr
b_{ij}+{|b_{ik}|b_{kj}+b_{ik}|b_{kj}|\over 2}  & \mbox{otherwise.}
\end{cases}
\end{eqnarray}
The seed $(\mathbf{x}', B')$ is said to be a  mutation of $(\mathbf{x}, B)$ in direction $k$,
where $1\le k\le n$;   the algebra  $\mathscr{A}(\mathbf{x}, B)$ is  generated by cluster  variables $\{x_i\}_{i=1}^{\infty}$
obtained from the initial seed $(\mathbf{x}, B)$ by the iteration of mutations  in all possible
directions $k$.   The Laurent phenomenon  says  that  $\mathscr{A}(\mathbf{x}, B)\subset \mathbf{Z}[\mathbf{x}^{\pm 1}]$,
where  $\mathbf{Z}[\mathbf{x}^{\pm 1}]$ is the ring of  the Laurent polynomials in  variables $\mathbf{x}=(x_1,\dots,x_n)$
depending on an initial seed $(\mathbf{x}, B)$.
 The  $\mathscr{A}(\mathbf{x}, B)$  is a commutative algebra with an additive abelian
semigroup consisting of the Laurent polynomials with positive coefficients. 
Thus the algebra  $\mathscr{A}(\mathbf{x}, B)$
is a countable abelian group with an order satisfying the Riesz interpolation property, i.e. a 
 dimension group. 
A  {\it cluster $C^*$-algebra}  $\mathbb{A}(\mathbf{x}, B)$  is   an  AF-algebra,  
such that 
\begin{equation}\label{eq2.2}
K_0(\mathbb{A}(\mathbf{x}, B))\cong  \mathscr{A}(\mathbf{x}, B),
\end{equation}
where $\cong$ is an isomorphism of the dimension groups \cite[Section 4.4.3]{Nik1}.

\subsection{Schottky uniformization of $\mathscr{D}$}
Consider the Riemann surface $\mathscr{D}$ (an annulus) defined 
by the formula (\ref{eq1.1}).  We shall use the Schottky uniformization 
of  $\mathscr{D}$ by the loxodromic transformations. 
Namely, let $\mathbf{C}P^1:=\mathbf{C} ~\cup ~\{\infty\}$ be the Riemann
sphere and consider the M\"obius transformation of $\mathbf{C}P^1$ 
given by the formula:
\begin{equation}\label{eq2.3}
z\longmapsto \mathbf{k} ~z, \quad\hbox{where} ~ z\in \mathbf{C}P^1 \quad\hbox{and}
\quad|\mathbf{k}|\neq 1. 
\end{equation}
It is easy to see, that (\ref{eq2.3}) can be written in the matrix form:
\begin{equation}\label{eq2.4}
A=
\left(
\begin{matrix} \sqrt{\mathbf{k}} & 0\cr 0 & {1\over  \sqrt{\mathbf{k}}}
\end{matrix}\right)
\in PSL_2(\mathbf{C}):=SL_2(\mathbf{C})/\pm I.
\end{equation}
It is well known, that (\ref{eq2.3}) is a loxodromic transformation if and only if
\begin{equation}\label{eq2.5}
tr^2~(A)={(\mathbf{k}+1)^2\over\mathbf{k}}\in\mathbf{C} ~\backslash ~[0,4]. 
\end{equation}
The {\it Schottky uniformization} of $\mathscr{D}$ is given by the formula:
\begin{equation}\label{eq2.6}
\mathscr{D}\cong \mathbf{C}P^1 / A^{\mathbf{Z}}. 
\end{equation}

\medskip
\subsection{Admissible values of  $tr^2 ~(A)$}
Recall that the moduli space of the annulus $\mathscr{D}$ is given
by the formula:
\begin{equation}\label{eq2.7}
T_{\mathscr{D}}=\left\{t={R\over r} ~|~ t>1\right\}.
\end{equation}
We consider a cluster $C^*$-algebra $\mathbb{A}(\mathscr{D})$ associated to 
a canonical triangulation of $\mathscr{D}$ 
[Fomin,  Shapiro  \& Thurston  2008]  \cite[Example 4.4]{FoShaThu1}.
It follows from the Sherman-Zelevinsky Theorem for the algebra $\mathbb{A}(\mathscr{D})$, 
that the admissible values  of the ``index'' ${(t+1)^2\over t}$  must belong to the set:
\begin{equation}\label{eq2.8}
[4, \infty) ~\bigcup  ~\{ 4\cos^2\left({\pi\over n}\right) ~|~n\ge 3\},
\end{equation}
 see [Glubokov \& Nikolaev 2018] \cite{GluNik1} for the proof. 

\medskip
We set $\mathbf{k}=t$ in the formulas (\ref{eq2.3}) - (\ref{eq2.5}).  Comparing (\ref{eq2.5}) 
and (\ref{eq2.8}), one gets 
\begin{equation}\label{eq2.9}
tr^2 ~(A)\in [4, \infty) ~\bigcup  ~\{ 4\cos^2\left({\pi\over n}\right) ~|~n\ge 3\}.
\end{equation}

\medskip
\begin{remark}\label{rmk1.2}
It follows from (\ref{eq2.5}) that $A$ is a  loxodromic transformation if and only 
if $tr^2~(A)\in (4,\infty)$.  The case $tr^2~(A)\in \{ 4\cos^2\left({\pi\over n}\right) ~|~n\ge 3\}$
corresponds to an elliptic transformation $A$ of order $n$. Finally, that case 
 $tr^2~(A)=4$ gives a parabolic transformation $A$.  Note that for the elliptic 
 and parabolic transformations, the values of parameter $t$ in (\ref{eq2.7}) are
  the  $n$-th roots of unity.  
 \end{remark}

\section{Proof of theorem \ref{thm1.1}}
We split the proof in a series of lemmas.
\begin{lemma}\label{lm3.1}
Let $p_1,p_2\in  \mathbf{C}P^1$ be two points on the Riemann sphere,
such that $p_1\ne p_2$. Then  there exists a double covering map
\begin{equation}\label{eq3.1}
 \mathbf{p} ~: ~\mathbf{C}P^1\to  \mathbf{C}P^1
 \end{equation}
 ramified over the points $p_1$ and $p_2$.
 \end{lemma}
\begin{proof}
(i) Recall that the necessary condition for the existence of $\mathbf{p}$ 
is given by the Riemann-Hurwitz formula:
\begin{equation}\label{eq3.2}
 \chi (\mathbf{C}P^1)=2 \chi (\mathbf{C}P^1)-\sum_{i=1}^2 (e_i-1),
 \end{equation}
where $\chi (\mathbf{C}P^1)$ is the Euler characteristic and $e_i$ is the 
degree of the map $z\mapsto z^{e_i}$ in  the ramification point $p_i$. 
Since $\chi (\mathbf{C}P^1)=e_1=e_2=2$, we conclude that the condition 
(\ref{eq3.2})  is satisfied. 

\medskip
(ii) The sufficient condition for the existence of $\mathbf{p}$ 
can be verified directly using [Gersten 1987] \cite[Theorem 1.5]{Ger1}. 
We leave it as an exercise to the reader. 

\medskip
Lemma \ref{lm3.1} is proved. 
\end{proof}

\medskip
\begin{remark}\label{rmk3.2}
We assume further that in lemma \ref{lm3.1} we have $p_1=0$ and $p_2=\infty$. 
Our assumption is not 
restrictive, since any two points $p_1,p_2\in\mathbf{C}P^1$ can be put
into such a position by a M\"obius transformation.   
\end{remark}

\begin{figure}[h]

\begin{tikzpicture}
\draw[line width=1pt][color=black] (0, 0) circle (2.5);
\draw[line width=1pt][color=black] (7.5, 0) circle (2.5);

\coordinate [label=left:$\rho$] (A) at (4,.5);
\coordinate [label=left:$\mathscr{D}$] (A) at (0.3,-3);
\coordinate [label=left:$\mathbb{H}/G_{\lambda}$] (B) at (8.5,-3);


\draw[fill=black] (0,-1.5) circle(.07) node[above] {$ \infty$};
\draw[fill=black] (7.5,-1.5) circle(.07) node[above] {$ e_p$};

\coordinate [label=left:$D(tr^{2}(A))$] (A) at (1,-.5);
\coordinate [label=left:$D(\lambda)$] (A) at (8,-.5);

\draw[line width=1pt][->] (3.2,0) -- (4.2,0);

\draw[line width=1pt][fill=black!15] (0,1) circle(1.0);
\draw[line width=1pt][fill=black!15] (7.5,1) circle(1.0);

\end{tikzpicture}

\caption{Covering map  $\rho :  \mathscr{D}\to  \mathbb{H}/G_{\lambda}$.}
\end{figure}

\begin{lemma}\label{lm3.3}
Consider a  Riemann surface:
\begin{equation}\label{eq3.3}
\mathscr{H}:=\mathbf{C}P^1 ~\backslash ~\{D_0, \infty\}, 
\end{equation}
where $D_0$ is a disk containing  point $0\in \mathbf{C}P^1$. 
Then there exists  a double covering map
\begin{equation}\label{eq3.4}
 \rho ~: ~\mathscr{D}\to  \mathscr{H}. 
 \end{equation}
\end{lemma}
\begin{proof}
Let  $\mathbf{p}: \mathbf{C}P^1\to \mathbf{C}P^1$ be the double covering  map 
ramified at $0$ and $\infty$, see lemma \ref{lm3.1} and remark \ref{rmk3.2}. 
Recall that 
\begin{equation}\label{eq3.5}
\mathscr{D}\cong \mathbf{C}P^1 ~\backslash ~\{D_0, D_{\infty}\}, 
\end{equation}
where $D_{\infty}$ is a disk containing point $\infty\in \mathbf{C}P^1$. 
We use a homotopy to contract $D_{\infty}$ to the point $\infty$,
and set the map $\rho\equiv \mathbf{p}$. Comparing (\ref{eq3.3}) and (\ref{eq3.5}),
we conclude that $\rho$ is the required double covering map. 
Lemma \ref{lm3.3} is proved.   
\end{proof}

\begin{lemma}\label{lm3.4}
$\mathscr{H} \cong \mathbb{H}/G_{\lambda}$, where 
$\lambda\in [2, \infty) ~\bigcup  ~\{ 2\cos \left({\pi\over n}\right) ~|~n\ge 3\}$.
\end{lemma}
\begin{proof}
It is well known, that the Hecke orbifold  $\{\mathbb{H}/G_{\lambda} ~|~ \lambda >2\}$
is a topological sphere $S^2$ with a hole $D(\lambda)$ of radius $\lambda/2$, one elliptic fixed point  $e_2$ of order 2 and one
puncture $c$, see e.g.  [Schmidt \& Sheingorn 1995] \cite[p. 255]{SchShe1}. 
The elliptic point $e_2=i$ is a fixed point of the matrix $\left(\begin{smallmatrix} 0 & 1\cr -1 & 0\end{smallmatrix}\right)$ 
having  order 2 in the group $SL_2(\mathbf{Z})/\pm I$. 

Since $S^2\cong \mathbf{C}P^1$, we use (\ref{eq3.3})  to  identify $D(\lambda)\equiv D_0$ and $e_2\equiv \infty$.
Thus one gets $\mathscr{H} \cong \mathbb{H}/G_{\lambda}$. 
 Using  lemma \ref{lm3.3}, we obtain  a double covering map $\rho$  ramified in the points  
 $e_2$ and $0\in D(\lambda)$, i.e. 
\begin{equation}\label{eq3.6}
\rho :  \mathscr{D}\to  \mathbb{H}/G_{\lambda}. 
\end{equation}
\begin{remark}
Notice that $\rho^{-1}(D(\lambda))$ is a disk and $\rho^{-1}(e_2)$ 
is a regular point of $\mathscr{D}$.  The corresponding ramification
points are shown in  Figure 1.
\end{remark}

\medskip
To determine admissible values of the moduli parameter $\lambda$, 
 observe that $\lambda={1\over \pi}|\partial D(\lambda)|$, where $\partial D(\lambda)$ 
 is the boundary of the disk $D(\lambda)$.  Observe that the local map at the point $0\in D(\lambda)$
 is given by the formula $z\mapsto z^2$.  Therefore using the polar coordinates,
  we conclude that:   
\begin{equation}\label{eq3.7}
\lambda^2={1\over\pi}\left|\rho^{-1} (\partial D(\lambda))\right|
\end{equation}
is a moduli parameter of the Riemann surface $\mathscr{D}$. 
But according to (\ref{eq1.2}) any such a parameter must coincide with the $tr^2 ~(A)$,
where $A$ is the matrix in the Schottky uniformization (\ref{eq1.1}).  
Taking positive values of the square root, one gets from (\ref{eq1.2})
\begin{equation}\label{eq3.8}
\lambda=tr~(A)\in  [2, \infty) ~\bigcup  ~\{ 2\cos \left({\pi\over n}\right) ~|~n\ge 3\}. 
\end{equation}

Lemma \ref{lm3.4} is proved. 
\end{proof}
\begin{remark}
For the sake of brevity,  lemma \ref{lm3.4} is proved for the continuous 
moduli $\lambda\in (2,\infty)$. The case of the discrete moduli is treated likewise,
see remark \ref{rmk1.2}. 
\end{remark}

\bigskip
Theorem \ref{thm1.1} follows from lemma \ref{lm3.4}.

\bibliographystyle{amsplain}


\end{document}